\documentclass[11pt]{amsart}
\usepackage{amsmath}
\usepackage[active]{srcltx}
\usepackage{t1enc}
\usepackage[latin2]{inputenc}
\usepackage{verbatim}
\usepackage{amsmath,amsfonts,amssymb,amsthm}
\usepackage[mathcal]{eucal}
\usepackage{enumerate}
\usepackage[centertags]{amsmath}
\usepackage{graphics}
\usepackage[active]{srcltx}

\newtheorem{theorem}{Theorem}[section]

\newtheorem{definition}{Definition}[section]
\newtheorem{corollary}{Corollary}[section]

\numberwithin{equation}{section}

\begin{document}
\title[On the Convergence and Summability of Fourier series ]{On the
Convergence and Summability of double Walsh-Fourier series of functions of
bounded generalized variation}
\author{Ushangi Goginava and Artur Sahakian}
\address{U. Goginava, Department of Mathematics, Faculty of Exact and
Natural Sciences, Ivane Javakhishvili Tbilisi State University, Chavcha\-vadze str. 1, Tbilisi
0128, Georgia}
\email{zazagoginava@gmail.com}
\address{A. Sahakian, Yerevan State University, Faculty of Mathematics and
Mechanics, Alex Manoukian str. 1, Yerevan 0025, Armenia}
\email{sart@ysu.am}
\date{}
\maketitle

\begin{abstract}
The convergence of partial sums and Ces\'aro means of negative order of
double Walsh-Fourier series of functions of bounded \ generalized variation
is investigated.
\end{abstract}

\footnotetext{%
2010 Mathematics Subject Classification 42C10 .

\par
The research of U. Goginava was supported by Shota Rustaveli National
Science Foundation grant no.31/48 (Operators in some function spaces and their applications in
Fourier analysis)
}

\section{ Classes of Functions of Bounded Generalized
Variation}

In 1881 Jordan \cite{Jo} introduced a class of functions of bounded
variation and applied it to the theory of Fourier series. Hereinafter this
notion was generalized by many authors (quadratic variation, $\Phi $%
-variation, $\Lambda $-variation ets., see \cite{Ch, M, Wi,Wa1}). In two
dimensional case the class BV of functions of bounded variation was
introduced by Hardy \cite{Ha}.

Let $f$ be a real and measurable function of two variables on the unit square. Given intervals $\Delta =(a,b)$, $J=(c,d)$ and points $x,y$ from $I:=[0,1)$ we denote
\begin{equation*}
f(\Delta ,y):=f(b,y)-f(a,y),\qquad f(x,J)=f(x,d)-f(x,c)
\end{equation*}%
and
\begin{equation*}
f(\Delta ,J):=f(a,c)-f(a,d)-f(b,c)+f(b,d).
\end{equation*}%
Let $E=\{\Delta _{i}\}$ be a collection of nonoverlapping intervals from $I$
ordered in arbitrary way and let $\Omega $ be the set of all such
collections $E$. Denote by $\Omega _{n}$ the set of all collections of $n$
nonoverlapping intervals $I_{k}\subset I$.

For the sequences of positive numbers
$$\Lambda ^{1}=\{\lambda
_{n}^{1}\}_{n=1}^{\infty },\qquad \Lambda ^{2}=\{\lambda
_{n}^{2}\}_{n=1}^{\infty }
$$
and $I^{2}:=[0,1)^{2}$ we denote
\begin{equation*}
\Lambda ^{1}V_{1}(f;I^{2})=\sup_{y}\sup_{E\in \Omega }\sum_{i}\frac{%
|f(\Delta _{i},y)|}{\lambda _{i}^{1}}\,\,\,\,\,\,\left( E=\{\Delta
_{i}\}\right) ,
\end{equation*}%
\begin{equation*}
\Lambda ^{2}V_{2}(f;I^{2})=\sup_{x}\sup_{F\in \Omega }\sum_{j}\frac{%
|f(x,J_{j})|}{\lambda _{j}^{2}}\qquad (F=\{J_{j}\}),
\end{equation*}%
\begin{equation*}
\left( \Lambda ^{1}\Lambda ^{2}\right) V_{1,2}(f;I^{2})=\sup_{F,\,E\in
\Omega }\sum_{i}\sum_{j}\frac{|f(\Delta _{i},J_{j})|}{\lambda
_{i}^{1}\lambda _{j}^{2}}.
\end{equation*}

\begin{definition}
We say that the function $f$ has Bounded $\left( \Lambda ^{1},\Lambda
^{2}\right) $-variation on $I^{2}$ and write $f\in \left( \Lambda
^{1},\Lambda ^{2}\right) BV\left( I^{2}\right) $, if
\begin{equation*}
\left( \Lambda ^{1},\Lambda ^{2}\right) V(f;I^{2}):=\Lambda
^{1}V_{1}(f;I^{2})+\Lambda ^{2}V_{2}(f;I^{2})+\left( \Lambda ^{1}\Lambda
^{2}\right) V_{1,2}(f;I^{2})<\infty .
\end{equation*}%
If $\Lambda^1=\Lambda^2=\Lambda$, we say $\Lambda$-variation and use the notation $\Lambda BV(I^2)$.

We say that the function $f$ has Bounded Partial $\Lambda $-variation and
write $f\in P\Lambda BV\left( I^{2}\right) $ if
\begin{equation*}
P\Lambda BV(f;I^{2}):=\Lambda V_{1}(f;I^{2})+\Lambda V_{2}(f;I^{2})<\infty .
\end{equation*}
\end{definition}

If $\Lambda=\{\lambda_n\}$ with $\lambda _{n}\equiv 1$ (or if $0<c<\lambda _{n}<C<\infty ,\ n=1,2,\ldots $%
) the classes $\Lambda BV$ and $P\Lambda BV$ coincide, respectively, with the Hardy class $BV$ and with the class $PBV$ functions  of bounded partial variation introduced by Goginava \cite{GoEJA}. Hence it is reasonable to
assume that $\lambda _{n}\rightarrow \infty $ and since the intervals in $%
E=\{\Delta _{i}\}$ are ordered arbitrarily, we will suppose, without loss of
generality, that the sequence $\{\lambda _{n}\}$ is increasing. Thus, we assume that
\begin{equation}
1<\lambda _{1}\leq \lambda _{2}\leq \ldots ,\qquad \lim_{n\rightarrow \infty
}\lambda _{n}=\infty, \qquad \sum\limits_{n=1}^{\infty }\left( 1/\lambda
_{n}\right) =+\infty .  \label{Lambda}
\end{equation}%

In the case when $\lambda _{n}=n,\ n=1,2\ldots $ we say \textit{Harmonic
Variation} instead of $\Lambda $-variation and write $H$ instead of $\Lambda$, i.e.
$HBV$, $PHBV$, $HV(f)$, ets.

The notion of $\Lambda $-variation was introduced by Waterman \cite{Wa1} in
one dimensional case, by Sahakian \cite{Saha} in two dimensional case. The
notion of bounded partial $\Lambda $-variation ($P\Lambda BV$) was
introduced by Goginava and Sahakian \cite{GogSah}.

Dyachenko and Waterman \cite{DW} introduced another class of functions of
generalized bounded variation. Denoting by $\Gamma $ the  set of finite
collections of nonoverlapping rectangles $A_{k}:=\left[ \alpha _{k},\beta
_{k}\right] \times \left[ \gamma _{k},\delta _{k}\right] \subset I^{2}$, we define
\begin{equation*}
\Lambda ^{\ast }V\left( f\right) :=\sup_{\{A_{k}\}\in \Gamma }\sum\limits_{k}%
\frac{\left\vert f\left( A_{k}\right) \right\vert }{\lambda _{k}}.
\end{equation*}

\begin{definition}[Dyachenko, Waterman]
Let $f$ be a real function on $I^{2}$. We say that $f\in \Lambda
^{\ast }BV$, if
\begin{equation*}
\Lambda V(f):=\Lambda V_{1}(f)+\Lambda V_{2}(f)+\Lambda ^{\ast }V\left(
f\right) <\infty .
\end{equation*}
\end{definition}

In \cite{GSGMJ}, the authors introduced a new classes of functions of
generalized bounded variation and investigate the convergence of Fourier
series of function of that classes.

For the sequence $\Lambda =\{\lambda _{n}\}_{n=1}^{\infty }$ we denote
\begin{equation*}
\Lambda ^{\#}V_{1}(f)=\sup_{\{y_{i}\}\subset I}\sup_{\{I_{i}\}\in \Omega
}\sum_{i}\frac{|f(I_{i},{y_{i}})|}{\lambda _{i}},
\end{equation*}%
\begin{equation*}
\Lambda ^{\#}V_{2}(f)=\sup_{\{x_{j}\}\subset I}\sup_{\{J_{j}\}\in \Omega
}\sum_{j}\frac{|f(x_{j},J_{j}|}{\lambda _{j}}.
\end{equation*}

\begin{definition}
We say that the function $f$ belongs to the class $\Lambda ^{\#}BV$, if
\begin{equation*}
\Lambda ^{\#}V(f):=\Lambda ^{\#}V_{1}(f)+\Lambda ^{\#}V_{2}(f)<\infty .
\end{equation*}
\end{definition}
The notion of continuity of function in $\Lambda$-variation plays an important role in the investigation of convergence Fourier series of functions of bounded $\Lambda$-variation.
\begin{definition}
We say that the function $f\,$\ is continuous in $\left( \Lambda
^{1},\Lambda ^{2}\right) $-variation on $I^{2}$ and write $f\in C\left(
\Lambda ^{1},\Lambda ^{2}\right) V$, if%
\begin{equation*}
\lim\limits_{n\rightarrow \infty }\Lambda _{n}^{1}V_{1}\left( f\right)
=\lim\limits_{n\rightarrow \infty }\Lambda _{n}^{2}V_{2}\left( f\right) =0
\end{equation*}%
and%
\begin{equation*}
\lim\limits_{n\rightarrow \infty }\left( \Lambda _{n}^{1},\Lambda
^{2}\right) V_{1,2}\left( f\right) =\lim\limits_{n\rightarrow \infty }\left(
\Lambda ^{1},\Lambda _{n}^{2}\right) V_{1,2}\left( f\right) =0,
\end{equation*}%
where $\Lambda _{n}^{i}:=\left\{ \lambda _{k}^{i}\right\} _{k=n}^{\infty
}=\left\{ \lambda _{k+n}^{i}\right\} _{k=0}^{\infty },\ i=1,2.$
\end{definition}

\begin{definition}\label{d5}
We say that the function $f\,$\ is continuous in $\Lambda ^{\#}$-variation
on $I^{2}$ and write $f\in C\Lambda ^{\#}V$, if%
\begin{equation*}
\lim\limits_{n\rightarrow \infty }\Lambda _{n}^{\#}V\left( f\right)=0
\end{equation*}%
where $\Lambda _{n}:=\left\{ \lambda _{k}\right\} _{k=n}^{\infty }.$
\end{definition}

\begin{definition}
We say that the function $f\,$\ is continuous in $\Lambda ^{\ast }$%
-variation on $I^{2}$ and write $f\in C\Lambda ^{\ast }V$, if%
\begin{equation*}
\lim\limits_{n\rightarrow \infty }\Lambda _{n}^{1}V_{1}\left( f\right)
=\lim\limits_{n\rightarrow \infty }\Lambda _{n}^{2}V_{2}\left( f\right) =0
\end{equation*}%
and%
\begin{equation*}
\lim\limits_{n\rightarrow \infty }\Lambda _{n}^{\ast }V\left( f\right) =0
\end{equation*}
\end{definition}

Now, we define
\begin{equation*}
v_{1}^{\#}\left( n,f\right)
:=\sup\limits_{\{y_{i}\}_{i=1}^{n}}\sup\limits_{\{I_{i}\}\in \Omega
_{n}}\sum\limits_{i=1}^{n}\left\vert f\left( I_{i},y_{i}\right) \right\vert
,\quad n=1,2,\ldots ,
\end{equation*}%
\begin{equation*}
v_{2}^{\#}\left( m,f\right)
:=\sup\limits_{\{x_{j}\}_{j=1}^{m}}\sup\limits_{\{J_{k}\}\in \Omega
_{m}}\sum\limits_{j=1}^{m}\left\vert f\left( x_{j},J_{j}\right) \right\vert
,\quad m=1,2,\ldots .
\end{equation*}

The following theorems hold.

\begin{theorem}[Goginava, Sahakian \cite{GSGMJ}]
\label{embedding1}$\left\{ \frac{n}{\log n}\right\} ^{\#}BV\subset HBV$.
\end{theorem}

\begin{theorem}[Goginava, Sahakian \cite{GSGMJ}]
\label{embedding2} Suppose
\begin{equation*}
\sum\limits_{n=1}^{\infty }\frac{v_{s}^{\#}\left( f;n\right) \log \left(
n+1\right) }{n^{2}}<\infty ,\quad s=1,2.
\end{equation*}%
Then $f\in \left\{ \frac{n}{\log \left( n+1\right) }\right\} ^{\#}BV$.
\end{theorem}

\begin{theorem}[Goginava \cite{GoSMJ}]
\label{embedding4}Let $\alpha ,\beta \in \left( 0,1\right) $,  $\alpha
+\beta <1$ and%
\begin{equation*}
\sum\limits_{j=1}^{\infty }\frac{v_{s}^{\#}\left( f;2^{j}\right) }{%
2^{j\left( 1-\left( \alpha +\beta \right) \right) }}<\infty ,\quad s=1,2.
\end{equation*}%
Then $f\in C\left\{ n^{1-\left( \alpha +\beta \right) }\right\} ^{\#}V$.
\end{theorem}

\begin{theorem}[Goginava \cite{GoSMJ}]\label {GO10}
\label{embedding3}Let $\alpha ,\beta \in \left( 0,1\right) $ and $\alpha
+\beta <1$. Then%
\begin{equation*}
C\left\{ i^{1-\left( \alpha +\beta \right) }\right\} ^{\#}V\subset C\left\{
i^{1-\alpha }\right\} \left\{ j^{1-\beta }\right\} V.
\end{equation*}
\end{theorem}
The next theorem shows, that for some sequences $\Lambda$ the classes $\Lambda^\#V$ and $C\Lambda^\#V$ coincide.
\begin{theorem}\label{continous}
Let the sequence $\Lambda=\left\{\lambda_n\right\}$ be as in (\ref{Lambda})  and
\begin{equation}\label{maincond}
\liminf_{n\to\infty}\frac{\lambda_{2n}} {\lambda_{n}}=q>1.
\end{equation}
Then $\Lambda^\#V=C\Lambda^\#V$.
\end{theorem}
\begin{proof}
Suppose to the contrary, that there exists a function $f\in \Lambda^\#V$ for which (see Definition \ref{d5})
$\liminf\limits_{n\rightarrow \infty }\Lambda _{n}^{\#}V\left( f\right)>0$.
Without loss of generality, we can assume that $\liminf\limits_{n\rightarrow \infty }\Lambda _{n}^{\# }V_1\left( f\right)=\delta>0$ and that $\delta=1$. Then, taking into account that the sequence
$\{\Lambda _{n}^{\# }V_1(f)\}$ is decreasing, we have
\begin{equation}\label {delta}
\lim\limits_{n\rightarrow \infty }\Lambda _{n}^{\# }V_1\left( f\right) =1.
\end{equation}
Let a natural $k$ and a numbers $\varepsilon>0$, $q_0\in(1,q)$ be fixed.

According to  (\ref{maincond}) and (\ref{delta}) there exist a natural $N^\prime>k$ such that
\begin{equation}\label{Nprime}
\frac{\lambda_{2n}} {\lambda_{n}}> q_0,\quad \Lambda _{n}^{\# }V\left( f\right)> 1- \varepsilon \quad\text{for}\quad n\geq N^\prime.
\end{equation}
Then for a natural $N>2N^\prime$ there are a set of points $\{y_i\}_{i=1}^{2i_0}$ and a set of nonoverlapping   intervals $\{\delta_i\}_{i=1}^{2i_0}\in \Omega$ such that
\begin{equation}\label{sumI}
I:=\sum_{i=1}^{2i_0}\frac {|f(\delta_i,y_i)|}{\lambda_{N+i}}\geq 1-\varepsilon.
\end{equation}
Adding, if necessary, new summands in (\ref{sumI}) we can assume that
$$
\bigcup_{i=1}^{2i_0}\delta_i=(0,1).
$$
Denote
\begin{equation}\label{sumI12}
I_1:=\sum_{i=1}^{i_0}\frac {|f(\delta_{2i-1},y_{2i-1})|}{\lambda_{N+2i-1}}, \qquad
I_2:=\sum_{i=1}^{i_0}\frac {|f(\delta_{2i},y_{2i})|}{\lambda_{N+2i}}.
\end{equation}
Since $N>2N^\prime$ implies that $N+2i-1\geq2(N^\prime+i)$, from (\ref{Nprime}) and  (\ref{sumI12}) we have
\begin{equation}\label{sumI1prime}
I_1^\prime:=\sum_{i=1}^{i_0}\frac {|f(\delta_{2i-1},y_{2i-1})|}{\lambda_{N^\prime+i}}=
\sum_{i=1}^{i_0}\frac {|f(\delta_{2i-1},y_{2i-1})|}{\lambda_{N+2i-1}}\cdot
\frac {\lambda_{N+2i-1}}{\lambda_{N^\prime+i}}>q_0I_1
\end{equation}
and
\begin{equation}\label{sumI2prime}
I_2^\prime:=\sum_{i=1}^{i_0}\frac {|f(\delta_{2i},y_{2i})|}{\lambda_{N^\prime+i}}=
\sum_{i=1}^{i_0}\frac {|f(\delta_{2i},y_{2i})|}{\lambda_{N+2i}}\cdot
\frac {\lambda_{N+2i}}{\lambda_{N^\prime+i}}>q_0I_2.
\end{equation}
Consequently, by (\ref{sumI}),
\begin{equation}\label{sumIprime}
I^\prime:=I_1^\prime+I_2^\prime\geq q_0(I_1+I_2)=q_0I\geq q_0(1-\varepsilon).
\end{equation}
Now, we take natural $M$ such that,
\begin{equation}\label{M}
M>N+2(i_0+1)\quad\text{and}\quad \frac {2(2i_0+1)}{\lambda_M}\sup_{x\in [0,1]}|f(x)|<\varepsilon,
\end{equation}
and using (\ref{Nprime}), we find a set of points $\{z_j\}_{j=1}^{j_0}$ set of nonoverlapping   intervals $\{\Delta_j\}_{j=1}^{j_0}\in \Omega$ such that
\begin{equation}\label{sum}
\sum_{j=1}^{j_0}\frac {|f(\Delta_j,z_j)|}{\lambda_{M+j}}\geq 1-\varepsilon.
\end{equation}
Denote by $Q$ the set of indices $j=1,2,\cdots,j_0$ for which the corresponding nterval $\Delta_j$ does not contain an endpoint of intervals $\delta_i,i=1,2,\ldots, 2i_0$, i.e. $\Delta_j$ lies in one of intervals $\delta_i,i=1,2,\ldots, 2i_0$. Then the number of indices in $[1,j_0]\setminus Q$ does not exceed $2i_0+1$ and by (\ref{M}),
\begin{equation*}
\sum_{j\in [1,j_0]\setminus Q}\frac {|f(\Delta_j,z_j)|}{\lambda_{M+j}}\leq \varepsilon.
\end{equation*}
Consequently, by (\ref {sum}),
\begin{equation}\label{sumJ}
J:=\sum_{j\in  Q}\frac {|f(\Delta_j,z_j)|}{\lambda_{M+j}}\geq 1-2 \varepsilon.
\end{equation}
Denoting
\begin{equation*}
Q_1=\left\{j\in Q: \Delta_j\subset \bigcup_{i=1}^{i_0}\delta_{2i-1}\right\},\qquad
Q_2=\left\{j\in Q: \Delta_j\subset \bigcup_{i=1}^{i_0}\delta_{2i}\right\}
\end{equation*}
and
\begin{equation*}
J_1:=\sum_{j\in  Q_1}\frac {|f(\Delta_j,z_j)|}{\lambda_{M+j}},\qquad
J_2:=\sum_{j\in  Q_2}\frac {|f(\Delta_j,z_j)|}{\lambda_{M+j}}
\end{equation*}
from (\ref{sumIprime}) and (\ref{sumJ}) we obtain
$$
(I_1^\prime+J_2) + (I_2^\prime+J_1)=I^\prime+J\geq q_0(1-\varepsilon)+1-2 \varepsilon \geq q_0+1 -3\varepsilon.
$$
Thereforore,
$$
I_1^\prime+J_2 \geq \frac {q_0+1-3\varepsilon}2\quad \text{or} \quad (I_2^\prime+J_1) \geq \frac {q_0+1-3\varepsilon}2,
$$
which means that
$$
\Lambda _{N^\prime}^{\# }V_1\left( f\right)\geq \frac {q_0+1-3\varepsilon}2
$$
and hence
$$
\Lambda _{k}^{\# }V_1\left( f\right)\geq \frac {q_0+1}2,
$$
since $\varepsilon$ is any positive number and $N^\prime>k$.
Taking into account that $k$ is an arbitrary natural number, the last inequality implies
\begin{equation*}
\lim\limits_{n\rightarrow \infty }\Lambda _{n}^{\# }V_1\left( f\right) \geq \frac {q_0+1}2>1,
\end{equation*}
which is a contradiction to the assumption (\ref{delta}). Theorem \ref{continous}  is proved.
\end{proof}
It is easy to see, that for any $\gamma>0$ the sequence $\lambda_n=n^\gamma,\ n=1,2,\ldots$ satisfies the condition (\ref{maincond}) with $q=2^\gamma$. Hence Theorem \ref{continous} implies
\begin {corollary}
If  $\, 0<\gamma\leq 1$, then
$\left\{ n^{\gamma}\right\} ^{\#}V=C\left\{ n^{\gamma}\right\} ^{\#}V$.
\end{corollary}
This, combined with Theorem \ref{GO10} implies
\begin{corollary}
\label{embedding3}Let $\alpha ,\beta \in \left( 0,1\right) $ and $\alpha
+\beta <1$. Then%
\begin{equation*}\label {contin}
\left\{ i^{1-\left( \alpha +\beta \right) }\right\} ^{\#}V\subset C\left\{
i^{1-\alpha }\right\} \left\{ j^{1-\beta }\right\} V.
\end{equation*}
\end{corollary}

\section{Walsh functions}

Let $\mathbb{P}$ be the set of positive integers, and $\mathbb{N}\mathbf{:=}%
\mathbb{P}\mathbf{\cup \{}0\mathbf{\}.}$ We denote the set of all integers by $\,\,%
\mathbb{Z}$ and the set of dyadic rational numbers in the unit interval $%
I:=[0,1)$ by $\mathbb{Q}$. Each element of $\mathbb{Q}$ is of
the form $\frac{p}{2^{n}}$ for some $p,n\in \mathbb{N},\,\,\,0\leq p\leq
2^{n}$. By a dyadic interval in $I$ we mean an interval of the form $%
I_{N}^{l}:=[l2^{-N},\left( l+1\right) 2^{-N})$ for some $l\in \mathbb{N}%
,0\leq l<2^{N}$. Given $N\in \mathbb{N}$ and $x\in I$, we denote by  $I_{N}\left(
x\right) $ the dyadic interval of length $2^{-N}$ that contains $x$. Finaly, we set  $I_{N}:=[0,2^{-N})$ and $\overline{I}_{N}:=I\backslash
I_{N}.$

Let $r_{0}\left( x\right) $ be the function defined on the real line by
\begin{equation*}
r_{0}\left( x\right) =
\begin{cases}
\ \ 1,&\mbox{ if }x\in [0,1/2) \\
-1,&\mbox{ if }x\in [1/2,1)
\end{cases},
\qquad r_{0}\left( x+1\right) =r_{0}\left( x\right),\quad x\in \mathbb R.
\end{equation*}
The Rademacher system is defined by
\begin{equation*}
r_{n}\left( x\right) =r_{0}\left( 2^{n}x\right) \quad x\in I,\qquad n=1,2,\ldots.
\end{equation*}

The Walsh functions $w_{0},w_{1},...$ are defined as follows. Denote $w_{0}\left(x\right) =1$ and if $\,\,k=2^{n_{1}}+\cdots +2^{n_{s}}\,$is a positive
integer with $n_{1}>n_{2}>\cdots >n_{s}$, then
\begin{equation*}
w_{k}\left( x\right) =r_{n_{1}}\left( x\right) \cdots r_{n_{s}}\left(
x\right) .
\end{equation*}
The Walsh-Dirichlet kernel is defined by
\begin{equation*}
D_{n}\left( x\right) =\sum\limits_{k=0}^{n-1}w_{k}\left( x\right),\quad  n=1,2,\ldots
\end{equation*}
Recall that \cite{GES, SWS}
\begin{equation}
D_{2^{n}}\left( x\right) =
\begin{cases}
2^{n},&\mbox{if }x\in \left[ 0,2^{-n}\right) \\
\, 0,&\mbox{if }x\in \left[ 2^{-n},1\right)
\end{cases}
\end{equation}
and%
\begin{equation}
D_{2^{n}+m}\left( x\right) =D_{2^{n}}\left( x\right) +w_{2^{n}}\left(
x\right) D_{m}\left( x\right) ,\quad 0\leq m<2^{n},\quad n=0, 1,\ldots  \label{dir}
\end{equation}
It is well known that \cite{SWS}
\begin{equation}
D_{n}\left( t\right) =w_{n}\left( t\right) \sum\limits_{j=0}^{\infty
}n_{j}w_{2^{j}}\left( t\right) D_{2^{j}}\left( t\right) , \quad \mbox{if} \quad
n=\sum_{j=0}^{\infty }n_{j}2^{j}  \label{w3}
\end{equation}
and%
\begin{equation}
\left\vert D_{q_{n}}\left( x\right) \right\vert \geq \frac{1}{4x}%
,\qquad 2^{-2n-1}\leq x<1,  \label{lowest}
\end{equation}%
where%
\begin{equation}\label{q_n}
q_{n}:=2^{2n-2}+2^{2n-4}+\cdots +2^{2}+2^{0}.
\end{equation}
Given $x\in I$, the expansion

\begin{equation}
x=\sum\limits_{k=0}^{\infty }x_{k}2^{-(k+1)},  \label{w4}
\end{equation}%
where each $x_{k}=0$ or $1$, is called a dyadic expansion of $x.$ If $%
x\in I\backslash \mathbb{Q}\mathbf{\,}$, then (\ref{w4}) is uniquely
determined. For $x\in \mathbb{Q}$ we choose the dyadic expansion
with $\lim\limits_{k\rightarrow \infty }x_{k}=0$.

The dyadic sum of $x,y\in I$ in terms of the dyadic expansion of $x$ and $y$
is defined by
\begin{equation*}
x\dotplus y=\sum\limits_{k=0}^{\infty }\left\vert x_{k}-y_{k}\right\vert
2^{-(k+1)}.
\end{equation*}
We say that $f\left( x,y\right) $ is continuous at $\left( x,y\right) $ if%
\begin{equation}
\lim\limits_{h,\delta \rightarrow 0}f\left( x\dotplus h,y\dotplus \delta
\right) =f\left( x,y\right) .  \label{cont}
\end{equation}

We consider the double system $\left\{ w_{n}(x)\times w_{m}(y):\,n,m\in
\mathbf{N}\right\} $ on the unit square $I^{2}=\left[ 0,1\right) \times %
\left[ 0,1\right) .$

If $f\in L^{1}\left( I^{2}\right) ,$ then
\begin{equation*}
\hat{f}\left( n,m\right) =\int\limits_{I^{2}}f\left( x,y\right)
w_{n}(x)w_{m}(y)dxdy
\end{equation*}%
is the $\left( n,m\right) $-th Walsh-Fourier coefficient of $f.$

The rectangular partial sums of double Fourier series with respect to the
Walsh system are defined by
\begin{equation*}
S_{M,N}(x,y;f)=\sum\limits_{m=0}^{M-1}\sum\limits_{n=0}^{N-1}\hat{f}\left(
m,n\right) w_{m}(x)w_{n}(y).
\end{equation*}

The Ces\`{a}ro $\left( C;\alpha ,\beta \right) $-means of double
Walsh-Fourier series are defined as follows
\begin{equation*}
\sigma _{n,m}^{\alpha ,\beta }(x,y;f)=\frac{1}{A_{n-1}^{\alpha
}A_{m-1}^{\beta }}\sum\limits_{i=1}^{n}\sum\limits_{j=1}^{m}A_{n-i}^{\alpha
-1}A_{m-j}^{\beta -1}S_{i,j}(x,y;f),
\end{equation*}%
where
\begin{equation*}
A_{0}^{\alpha }=1,\quad A_{n}^{\alpha }=\frac{\left( \alpha +1\right) \cdots
\left( \alpha +n\right) }{n!},\quad \alpha \neq -1,-2,....
\end{equation*}

It is well-known that \cite{Zy}%
\begin{equation}
A_{n}^{\alpha }=\sum\limits_{k=0}^{n}A_{n-k}^{\alpha -1},  \label{f1}
\end{equation}%
\begin{equation}
A_{n}^{\alpha }\sim n^{\alpha }  \label{f2}
\end{equation}%
and%
\begin{equation}
\sigma _{n,m}^{\alpha ,\beta }(x,y;f)=\int\limits_{I^{2}}f\left( s,t\right)
K_{n}^{\alpha }\left( x\dotplus s\right) K_{m}^{\beta }\left( y\dotplus
t\right) dsdt,  \label{int}
\end{equation}%
where%
\begin{equation}\label {kernel}
K_{n}^{\alpha }\left( x\right) :=\frac{1}{A_{n-1}^{\alpha }}%
\sum\limits_{k=1}^{n}A_{n-k}^{\alpha -1}D_{k}\left( x\right) .
\end{equation}

\section{Convergence of two-dimensional Walsh-Fourier series}

The well known Dirichlet-Jordan theorem (see \cite{Zy}) states that the
Fourier series of a function $f(x),\ x\in T$ of bounded variation converges
at every point $x$ to the value $\left[ f\left( x+0\right) +f\left(
x-0\right) \right] /2$.

Hardy \cite{Ha} generalized the Dirichlet-Jordan theorem to the double
Fourier series. He proved that if function $f(x,y)$ has bounded variation in
the sense of Hardy ($f\in BV$), then $S\left[ f\right] $ converges at any
point $\left( x,y\right) $ to the value $\frac{1}{4}\sum f\left( x\pm 0,y\pm
0\right) $. Here and below we consider the convergence of {\bf only rectangular partial sums} of double Fourier series.

Convergence of $d$-dimensional trigonometric Fourier series of functions of
bounded $\Lambda $-variation was investigated in details by Sahakian \cite%
{Saha}, Dyachenko \cite{D1, D2, DW}, Bakhvalov \cite{Bakh1}, Sablin \cite%
{Sab}, Goginava, Sahakian \cite{GogSah,GSGMJ}, ets.

For the $d$-dimensional Walsh-Fourier series the convergence of partial sums
of functions of bounded Harmonic variation and other bounded generalized
variation were studied by Moricz \cite{Mo, Mo2}, Onnewer, Waterman \cite{OW},
Goginava \cite{GoAMH}.

In the two-dimensional case Sargsyan  has obtained the following result.
\begin{theorem}[{Sargsyan \cite%
{Sar}}]\label{ori}
\label{d=2}If $f\in HBV(I^{2})$, then the double Walsh-Fourier series
of $f$ converges to $f\left( x,y\right) $ at any point $\left( x,y\right)
\in I^{2}$, where $f$ is continuous.
\end{theorem}

The authors investigated
convergence of multiple Walsh-Fourier series of functions of  partial $\Lambda $-bounded variation. In particular, the following result was proved.
\begin{theorem}[Goginava, Sahakian \cite{GS2012}]
\label{c3} a) If $f\in P\{\frac{n}{\log ^{1+\varepsilon }n}\}BV(I^{2})$ for
some $\varepsilon >0$, then the double Walsh-Fourier series of $f$
 converges to $f\left( x,y\right) $ at any point $\left( x,y\right) $%
, where $f$ is continuous.

b) There exists a continuous function $f\in P\{\frac{n}{\log n}%
\}BV(I^{2})$ such that the quadratic partial sums of its
Walsh-Fourier series diverge at some point.
\end{theorem}
In the next theorem we obtain a similar result for functions of bounded
$\Lambda^\#$-variation.
\begin{theorem}
\label{conv}a)
If $f\in \left\{ \frac{n}{\log n}\right\}^\# BV$, then
the double Walsh-Fourier series of $f$ converges to $f\left(
x,y\right) $ at any point $\left( x,y\right) $, where $f$ is continuous.%
\newline\indent
b) For an arbitrary sequence $\alpha _{n}\rightarrow \infty $  there
exists a continuous function $f\in \left\{ \frac{n\alpha _{n}}{\log \left(
n+1\right) }\right\} ^{\#}BV$ such that
 the quadratic partial sums of its
Walsh-Fourier series diverge unboundedly at $(0,0)$.
\end{theorem}

\begin{proof}
Part (a) immediately follows from Theorems \ref{embedding1} and  \ref{ori}.

To
prove part (b) observe  that  for any sequence $\Lambda =\{\lambda _{n}\}$
satisfying (\ref{Lambda}) the class $C\left( I^{2}\right) \bigcap \Lambda ^{\#}BV$ is a
Banach space with the norm
\begin{equation*}
\left\Vert f\right\Vert _{\Lambda ^{\#}BV}:=\left\Vert f\right\Vert
_{C}+\Lambda ^{\#}BV\left( f\right) ,
\end{equation*}
and $S_{N,N}(0,0,f)$, $n=1,2,\ldots$, is a sequence of bounded linear functionals on that space.
Denote
\begin{eqnarray}\label{phi_N}
\hskip-3mm \varphi _{N,j}\left( x\right) \hskip -2mm &=&\hskip -2mm 
\begin{cases}
\ 2^{2N+1}x-2j, &\text{if }x\in \left[ j2^{-2N},\left( 2j+1\right) 2^{-2N-1}%
\right]  \nonumber \\
-\left( 2^{2N+1}x-2j-2\right), &\text{if }x\in \left[ \left( 2j+1\right)
2^{-2N-1},\left( j+1\right) 2^{-2N}\right]  \\
\ 0, &\text{if }x\in I\backslash \left[ j2^{-2N},\left( j+1\right) 2^{2N}\right]
\end{cases},\\
&&\hskip-3mm \varphi _{N}\left( x\right) =\sum\limits_{j=1}^{2^{2N}-1}\varphi
_{N,j}\left( x\right) ,\qquad x\in I,\\
 g_{N}\left( x,y\right) \hskip -2mm &=&\hskip -2mm \varphi _{N}\left( x\right) \varphi _{N}\left(
y\right) \text{sgn}D_{q_{N}}\left( x\right) \text{sgn}D_{q_{N}}\left(
y\right),\qquad x,y\in I, \nonumber 
\end{eqnarray}
where $q_N$ is defined in (\ref{q_n}).

Suppose $\Lambda= \left\{\lambda_n= \frac{n\alpha_{n}}{\log \left( n+1\right) }\right\}_{n=1}^\infty$, where  $\alpha _{n}\rightarrow \infty $. It is easy to show that
$$
\Lambda^{\#}V_{s}\left(
g_{N}\right)
\leq c\sum\limits_{i=1}^{2^{2N}-1}\frac{\log \left( i+1\right) }{i\alpha _{i}}
\\
=o\left( N^{2}\right) \text{ as \ }N\rightarrow \infty ,
$$
for $s=1,2$. Hence
\begin{equation*}
\left\Vert g_{N}\right\Vert _{\Lambda ^{\#}BV}=o\left( N^{2}\right) =\eta
_{N}N^{2},
\end{equation*}%
where $\eta _{N}\to 0 $ as \ $N\rightarrow \infty $, and denoting
\begin{equation*}
G_{N}:=\frac{g_{N}}{\eta _{N}N^{2}},
\end{equation*}%
we conclude that $G_{N}\in \Lambda
^{\#}BV$ and
\begin{equation}\label{GN}
\sup\limits_{N}\left\Vert G_{N}\right\Vert _{\Lambda ^{\#}BV}<\infty .
\end{equation}

By construction of the function $G_{N}$  we have
\begin{eqnarray}\label{III}
S_{q_{N},q_{N}}\left( 0,0;G_{N}\right)
&=&
\iint\limits_{I^{2}}G_{N}\left( x,y\right) D_{q_{N}}\left( x\right)
D_{q_{N}}\left( y\right) dxdy\nonumber\\
&=&
\frac{1}{N^{2}\eta _{N}}\iint\limits_{I^{2}}\varphi _{N}\left( x\right)
\varphi _{N}\left( y\right) \left\vert D_{q_{N}}\left( x\right) \right\vert
\left\vert D_{q_{N}}\left( y\right) \right\vert dxdy\\
&=&
\frac{1}{N^{2}\eta _{N}}\left( \int\limits_{I}\varphi _{N}\left( x\right)
\left\vert D_{q_{N}}\left( x\right) \right\vert dx\right) ^{2}\nonumber
\end{eqnarray}
Using (\ref{lowest}) we can write
\begin{eqnarray*}
\int\limits_{I}\varphi _{N}\left( x\right) \left\vert D_{m_{N}}\left(
x\right) \right\vert dx
&=&\sum\limits_{j=1}^{2^{2N}-1}\int\limits_{j2^{-2N}}^{\left( j+1\right)
2^{-2N}}\varphi _{N,j}\left( x\right) \left\vert D_{m_{N}}\left( x\right)
\right\vert dx
\\
&=&\sum\limits_{j=1}^{2^{2N}-1}\left\vert D_{m_{N}}\left( \frac{j}{2^{2N}}%
\right) \right\vert \int\limits_{j2^{-2N}}^{\left( j+1\right)
2^{-2N}}\varphi _{N,j}\left( x\right) dx
\\
&\geq&
 \frac{1}{2^{2N+1}}\sum\limits_{j=1}^{2^{2N}-1}\frac{2^{2N}}{4j}\geq cN.
\end{eqnarray*}%
Consequently, from (\ref{III}) we obtain
\begin{equation}
\left\vert S_{q_{N},q_{N}}\left( 0,0;G_{N}\right) \right\vert \geq \frac{c}{%
\eta _{N}}\rightarrow \infty\quad  \mbox{as}\quad  N\rightarrow \infty. \label{BS}
\end{equation}%
According to the Banach-Steinhaus Theorem, (\ref{GN}) and (\ref{BS}) imply that there
exists a continuous function $f\in \left\{ \frac{n\alpha _{n}}{\log \left(
n+1\right) }\right\} ^{\#}BV$ such that
\begin{equation*}
\sup\limits_{N}\left\vert S_{N,N}\left( 0,0;f\right) \right\vert =+\infty .
\end{equation*}
Theorem \ref{conv} is proved.
\end{proof}

Theorem \ref{embedding2} and Theorem \ref{conv} imply

\begin{theorem}
Let the function $f\left( x,y\right)$, $(x,y)\in I^2$, satisfies the condition
\begin{equation*}
\sum\limits_{n=1}^{\infty }\frac{v_{s}^{\#}\left( f,n\right) \log \left(
n+1\right) }{n^{2}}<\infty ,~\ \ s=1,2.
\end{equation*}%
Then the double Walsh-Fourier series of $f$ converges to $%
f\left( x,y\right) $ at any point $\left( x,y\right) $, where $f$ is
continuous.
\end{theorem}

\section{  Ces\'aro means of negative order two-dimensional
Walsh-Fourier series}

The problems of summability of  Ces\'aro means of negative order for one
dimensional Walsh-Fourier series were studied in the works \cite{Tev}, \cite%
{GoJAT2}. In the two-dimensional case the summability of  Walsh-Fourier series by Ces\'aro metod of negative order for functions of partial bounded variation was
investigated by the first author author in \cite{GoJAT}, \cite{GoUMJ}. In
particular, the following results were obtained.

\begin{theorem}[{Goginava} \protect\cite{GoJAT}]
Let $f\in C_w\left( I^{2}\right) \cap PBV$ and $\alpha +\beta <1,\ \alpha
,\beta >0.$ Then the double Walsh-Fourier series of the function $f$ is
uniformly $(C;-\alpha ,-\beta )$ summable in the sense of Pringsheim.
\end{theorem}

\begin{theorem}[{Goginava} \protect\cite{GoJAT}]
Let $\alpha +\beta \geq 1,\ \alpha ,\beta >0.$ Then there exists a
continuous function $f_{0}\in PBV$ such that the Ces\`{a}ro $(C;-\alpha
,-\beta )$ means $\sigma _{n,n}^{-\alpha ,-\beta }\left( 0,0;f_{0}\,\right) $
of the doubleWalsh-Fourier series of $f_{0}$ diverges.
\end{theorem}

\begin{theorem}[Goginava \cite{GoUMJ}]
\label{(c,a,b)-conv}Let $f\in C\left( \left\{ i^{1-\alpha }\right\} ,\left\{
i^{1-\beta }\right\} \right) V\left( I^{2}\right) ,\alpha ,\beta \in \left(
0,1\right) $. Then $\left( C,-\alpha ,-\beta \right) $- means of double
Walsh-Fourier series converges to $f\left( x,y\right) $, if $f$ is
continuous at $\left( x,y\right) $.
\end{theorem}

\begin{theorem}[Goginava \cite{GoUMJ}]
Let $\alpha ,\beta \in \left( 0,1\right) ,\,\alpha +\beta <1$. \newline
a)If $f\in P\left\{ \frac{n^{1-\left( \alpha +\beta \right) }}{\log
^{1+\varepsilon }\left( n+1\right) }\right\} BV(I^{2})$ for some $%
\varepsilon >0$, then the double Walsh-Fourier series of the function $f$ is
$\left( C;-\alpha ,-\beta \right) $ summable to $f\left( x,y\right) $, if $f$
is continuous at $\left( x,y\right) $.\newline
b) There exists a continuous function $f\in P\left\{ \frac{n^{1-\left(
\alpha +\beta \right) }}{\log \left( n+1\right) }\right\} BV(I^{2})$ such
that $\sigma _{2^{n},2^{n}}^{-\alpha ,-\beta }\left( 0,0;f\right) $ diverges.
\end{theorem}

In this paper we prove that the following are true.

\begin{theorem}
\label{Main}a) Let $\alpha ,\beta \in \left( 0,1\right) ,\,\alpha +\beta <1$
and $f\in \left\{ n^{1-\left( \alpha +\beta \right) }\right\} ^{\#}BV$.
Then $\sigma _{n,m}^{-\alpha ,-\beta }\left( x,y;f\right) $ converges to $%
f\left( x,y\right) $, if $f$ is continuous at $\left( x,y\right) $. \newline
b) Let $\Lambda :=\left\{ n^{1-\left( \alpha +\beta \right) }\xi
_{n}\right\} $, where $\xi _{n}\uparrow \infty $ as $n\rightarrow \infty $.
Then \ there exists a function $f\in C\left( I^{2}\right) \cap C\Lambda
^{\#}V$ for which $\left( C;-\alpha ,-\beta \right) $-means of double
Walsh-Fourier series diverge unboundedly at $\left( 0,0\right) $.
\end{theorem}

\begin{proof}
Part a) immediately follows from, Corollary \ref{contin} and Theorem \ref%
{(c,a,b)-conv} .

To prove part b) observe that
\begin{equation*}
\left\{ n^{1-\left( \alpha +\beta \right) }\sqrt{\xi _{n}}\right\}
^{\#}BV\subset C\left\{ n^{1-\left( \alpha +\beta \right) }\xi _{n}\right\}
^{\#}V,
\end{equation*}%
and since $\xi _{n}\uparrow \infty $ is arbitrary, it is enough to show that there exists a continuous function $f\in \Lambda ^{\#}BV$ for which $\left(
C;-\alpha ,-\beta \right) $-means of double Walsh-Fourier series diverges
unboundedly at $\left( 0,0\right) $.

Denote
\begin{equation*}
h_{N}\left( x,y\right) :=\varphi _{N}\left( x\right) \varphi _{N}\left(
y\right) \text{sgn}K_{2^{2N}}^{-\alpha }\left( x\right) \text{sgn}%
K_{2^{2N}}^{-\beta }\left( y\right), 
\end{equation*}
where $\varphi_N$ is defined in (\ref{phi_N}), and the kernel $K_n^\alpha$ is  defined in (\ref {kernel}). 
It is easy to show that for $s=1,2$, 
\begin{eqnarray*}
\left\{ n^{1-\left( \alpha +\beta \right) }{\xi _{n}}\right\}
^{\#}V_{s}\left( h_{N}\right)  
&\leq &c\left( \alpha ,\beta \right) \sum\limits_{i=1}^{2^{2N}-1}\frac{1}{%
i^{1-\left( \alpha +\beta \right) }{\xi _{i}}} \\
&=&o\left( 2^{2N\left( \alpha +\beta \right) }\right), \text{ as \ }%
N\rightarrow \infty ,
\end{eqnarray*}%
hence 
\begin{equation*}
\left\Vert h_{N}\right\Vert _{\Lambda ^{\#}BV}=o\left( 2^{2N\left( \alpha
+\beta \right) }\right) =:\eta _{N}2^{2N\left( \alpha +\beta \right) },
\end{equation*}%
where $\eta _{N}=o\left( 1\right) $ as \ $N\rightarrow \infty $. Consequently, denoting
\begin{equation*}
H_{N}\left( x,y\right) :=\frac{h_{N}\left( x,y\right) }{\eta _{N}2^{2N\left(
\alpha +\beta \right) }},
\end{equation*}%
we conclude that $H_{N}\in C(I^2)\cap\Lambda ^{\#}BV$ and 
\begin{equation}\label{bound}
\sup\limits_{N}\left\Vert H_{N}\right\Vert _{\Lambda ^{\#}BV}<\infty .
\end{equation}%
By construction of the function $H_N$, we have
\begin{eqnarray}
&&\hskip-15mm \sigma _{2^{2N},2^{2N}}^{-\alpha ,-\beta }\left( 0,0;H_{N}\right)
\label{(c,a,b)}=
\iint\limits_{I^{2}}H_{N}\left( x,y\right) K_{2^{2N}}^{-\alpha }\left(
x\right) K_{2^{2N}}^{-\beta }\left( y\right) dxdy\nonumber\\
&=&
\frac{1}{\eta _{N}2^{2N\left( \alpha +\beta \right) }}\iint%
\limits_{I^{2}}h_{N}\left( x,y\right) K_{2^{2N}}^{-\alpha }\left( x\right)
K_{2^{2N}}^{-\beta }\left( y\right) dxdy\\
&=&\frac{1}{\eta _{N}2^{2N\left( \alpha +\beta \right) }}\int\limits_{I}%
\varphi _{N}\left( x\right) \left\vert K_{2^{2N}}^{-\alpha }\left( x\right)
\right\vert dx\int\limits_{I}\varphi _{N}\left( y\right) \left\vert
K_{2^{2N}}^{-\beta }\left( y\right) \right\vert dy.\nonumber
\end{eqnarray}%
Now, using the following estimate from \cite{Tev}:
\begin{equation*}
\int\limits_{2^{m-N-1}}^{2^{m-N}}\left\vert K_{2^{N}}^{-\alpha }\left(
x\right) \right\vert dx\geq c\left( \alpha \right) 2^{m\alpha },\quad N\in \mathbb{%
N},\quad m=1,...,N,\quad 0<\alpha <1,
\end{equation*}%
we can write%
\begin{equation}\label{a}
\int\limits_{I}\varphi _{N}\left( x\right) \left\vert K_{2^{2N}}^{-\alpha
}\left( x\right) \right\vert dx  
=\sum\limits_{j=1}^{2^{2N}-1}\int\limits_{j2^{-2N}}^{\left( j+1\right)
2^{-2N}}\varphi _{N,j}\left( x\right) \left\vert K_{2^{2N}}^{-\alpha }\left(
x\right) \right\vert dx
\end{equation}%
\begin{equation*}
=\sum\limits_{j=1}^{2^{2N}-1}\left\vert K_{2^{2N}}^{-\alpha }\left( \frac{j}{%
2^{2N}}\right) \right\vert \int\limits_{j2^{-2N}}^{\left( j+1\right)
2^{-2N}}\varphi _{N,j}\left( x\right) dx
\end{equation*}%
\begin{equation*}
=\frac{1}{2}\sum\limits_{j=1}^{2^{2N}-1}\left\vert K_{2^{2N}}^{-\alpha
}\left( \frac{j}{2^{2N}}\right) \right\vert \int\limits_{j2^{-2N}}^{\left(
j+1\right) 2^{-2N}}dx
\end{equation*}%
\begin{equation*}
=\frac{1}{2}\sum\limits_{j=1}^{2^{2N}-1}\int\limits_{j2^{-2N}}^{\left(
j+1\right) 2^{-2N}}\left\vert K_{2^{2N}}^{-\alpha }\left( x\right)
\right\vert dx
=\frac{1}{2}\sum\limits_{m=0}^{2N-1}\sum\limits_{j=2^{m}}^{2^{m+1}-1}\int%
\limits_{j2^{-2N}}^{\left( j+1\right) 2^{-2N}}\left\vert K_{2^{2N}}^{-\alpha
}\left( x\right) \right\vert dx
\end{equation*}%
\begin{equation*}
=\frac{1}{2}\sum\limits_{m=0}^{2N-1}\int\limits_{2^{m-2N}}^{2^{m+1-2N}}\left%
\vert K_{2^{2N}}^{-\alpha }\left( x\right) \right\vert dx
\geq c\left( \alpha \right) \sum\limits_{m=0}^{2N-1}2^{m\alpha }\geq c\left(
\alpha \right) 2^{2N\alpha }.
\end{equation*}
Analogously, we can prove that%
\begin{equation}
\int\limits_{I}\varphi _{N}\left( x\right) \left\vert K_{2^{2N}}^{-\beta
}\left( x\right) \right\vert dx\geq c\left( \beta \right) 2^{2N\beta },\quad N\in
\mathbb{N},\quad 0<\beta <1.  \label{b}
\end{equation}
Combining (\ref{a}) and (\ref{b}) we get
\begin{equation}
\left\vert \sigma _{2^{2N},2^{2N}}^{-\alpha ,-\beta }\left( 0,0;H_{N}\right)
\right\vert \geq \frac{c\left( \alpha ,\beta \right) }{\eta _{N}}\rightarrow
\infty \text{ as }N\rightarrow \infty \text{.}  \label{BS2}
\end{equation}

Applying the Banach-Steinhaus Theorem, from (\ref {bound}) and (\ref{BS2}) we obtain that there
exists a continuous function $f\in \Lambda ^{\#}BV$ such that
\begin{equation*}
\sup\limits_{N}\left\vert \sigma _{N,N}^{-\alpha ,-\beta }\left(
0,0,;f\right) \right\vert =+\infty .
\end{equation*}
Theorem \ref{Main} is proved.
\end{proof}

Since%
\begin{equation*}
\Lambda ^{\ast }BV\subset \Lambda ^{\#}BV
\end{equation*}%
from Theorem \ref{Main} we conclude that the following is true.

\begin{corollary}
Let $\alpha ,\beta \in \left( 0,1\right) ,\,\alpha +\beta <1$ and $f\in
\left\{ n^{1-\left( \alpha +\beta \right) }\right\} ^{\ast }BV$. Then $\sigma _{n,m}^{-\alpha ,-\beta }\left( x,y;f\right) $ converges to $%
f\left( x,y\right) $, if $f$ is continuous at $\left( x,y\right) $.
\end{corollary}

Theorem \ref{Main} and Theorem \ref{embedding4} imply.

\begin{theorem}
Let $\alpha ,\beta \in \left( 0,1\right) ,\,\alpha +\beta <1$ and%
\begin{equation*}
\sum\limits_{j=1}^{\infty }\frac{v_{s}^{\#}\left( f;2^{j}\right) }{%
2^{j\left( 1-\left( \alpha +\beta \right) \right) }}<\infty ,\quad s=1,2.
\end{equation*}%
Then $\sigma _{n,m}^{-\alpha ,-\beta }\left( x,y;f\right) $ converges to $%
f\left( x,y\right) $, if $f$ is continuous at $\left( x,y\right) $.
\end{theorem}

\end{document}